\documentclass[12pt]{amsart}
\usepackage{upkg}

\title{Coxeter and Schubert combinatorics of $\mu$--Involutions}

\author{Jack Chen-An Chou}
\address{Jack Chen-An Chou, Department of Mathematics, University of Florida, Gainesville, FL 32611.}
\email{c.chou@ufl.edu}

\author{Zachary Hamaker}
\address{Zachary Hamaker, Department of Mathematics, University of Florida, Gainesville, FL 32611.}
\email{zhamaker@ufl.edu}

\begin{document}

\begin{abstract}
    The variety of complete quadrics is the wonderful compactification of $GL_n/O_n$ and admits a cell decomposition into Borel orbits indexed by combinatorial objects called $\mu$--involutions.
    We study Coxeter–theoretic properties of $\mu$--involutions with results including a combinatorial description for their atoms, an exchange lemma, and transposition-like operators that characterize their Bruhat order.
    The corresponding orbit closures can be realized inside the flag variety.
    In this setting, we study the cohomology representatives of these orbits, which are, up to a scalar, the $\mu$--involution Schubert polynomials.
    We expand $\mu$--involution Schubert polynomials as a multiplicity-free sum of $\nu$--involution Schubert polynomials when $\nu$ refines $\mu$ and provide recurrences analogous to Monk's rule for Schubert polynomials.
\end{abstract}

\maketitle

\section{Introduction}

The complete flag variety $\text{FL}_n = GL_n/B$ has a natural stratification given by Borel orbit closures, indexed by permutations, whose classes form a basis for its cohomology ring.
This ring is the modern setting for Schubert calculus, an area of algebraic geometry that studies enumerative aspects of intersection theory.
Schubert polynomials, introduced by Lascoux and Sch{\"u}tzenberger~\cites{ls82a}, represent the cohomology classes of these orbit closures and provide a combinatorial model for this theory.
Their rich structure can be seen through the lens of geometry, representation theory, and combinatorics.

Significant effort has been devoted to extending Schubert calculus beyond the setting of the Schubert stratification for $\text{Fl}_n$.
One important such extension is the symmetric space given by orthogonal group orbit closures on $\text{Fl}_n$, which are indexed by involutions in the symmetric group.
These were initially studied by Richardson and Springer~\cites{richardson1990bruhat,Richardson1993}, with polynomial representatives introduced by Wyser and Yong in~\cite{wyser2017polynomials}.
These representatives, divided by a predictable power of 2, are the \emph{involution Schubert polynomials}.
A general theorem of Brion~\cite{brion1998behaviour} shows involution Schubert polynomials are multiplicity-free sums of Schubert polynomials.
The terms in this sum, which we call \emph{atoms}, have a combinatorial description due to Can, Joyce and Wyser~\cite{CJW} (see also~\cite{HMP2}).
Hamaker, Marberg, and Pawlowski~\cite{hamaker2018transition} demonstrate a recurrence relation on involution Schubert polynomials analogous to Monk's rule.

Orthogonal orbits of $\text{Fl}_n$ can be viewed as Borel orbits on $GL_n/O_n$.
This latter space is not compact, but has a wonderful compactification as the variety of complete quadrics $\text{Q}_n$~\cite{de2006complete}.
In this paper, we study the Borel orbits of $\text{Q}_n$ and their realizations as subvarieties of $\text{Fl}_n$.
Passing from $GL_n/O_n$ to $\text{Q}_n$ introduces boundary strata corresponding to partial degenerations of a quadric, and these strata are indexed by compositions of $n$.
In the stratum indexed by $\mu = (\mu_1,\dots,\mu_k)$ the Borel orbits correspond to  $H_\mu$ orbits of $\text{Fl}_n$ where $H_\mu$ is the semidirect product of $O_{\mu_1} \times \dots \times O_{\mu_k}$ and the parabolic subgroup $P_\mu \subseteq GL_n$.
These orbits are indexed by \emph{$\mu$--involutions}, which are permutations $\pi$ expressible as a concatenation of blocks $B_1, \dots, B_k$ with $|B_i| = \mu_i$ such that each $B_i$ has the relative order of an involution.
Since $H_\mu$ interpolates between $O_n$ when $\mu = (n)$ and $B$ when $\mu = (1^n)$, $\mu$--involutions interpolate between involutions and permutations.

Orbit closure containment defines a Bruhat order on $\mu$--involutions as $\mu$ varies, which is a special case of a more general family of orders studied in~\cite{timashev1994generalization}.
For fixed $\mu$, there is also a weak order studied in~\cite{can2013weak}, which determines analogues of reduced expressions for $\mu$--involutions.
In~\cite{banerjee2016combinatorial}, the authors show Bruhat order on $\mu$--involutions is determined by subword containment (see Theorems~\ref{t:mu-bruhat-def} and~\ref{t:mu-nu-cover}).
Using this realization, we provide a novel characterization of Bruhat order on $\mu$--involutions using transposition-like operators $t^\mu_{ab}$ (see Section~\ref{ss:transposition} and specifically Theorem~\ref{t:mu-atom-transposition}).

The Brion results extend to this setting, so the classes of $H_\mu$ orbit closures in $\text{Fl}_n$ can be represented as sums of Schubert polynomials.
Again rescaling by a predictable power of 2, for each $\mu$--involution $\pi$ we obtain the \emph{$\mu$--involution Schubert polynomial} $\fS^\mu_\pi$ as a multiplicity-free sum of Schubert polynomials.
The set of atoms in this expansion is denoted $\mA_\mu(\pi)$.
These polynomials can also be defined using divided difference operators applied to $\fS^\mu_{w_0}$, which is computed geometrically as a product of monomials and binomials in~\cite{can2015wonderful} or as a sum over $\mA_\mu(w_0)$ in~\cite{can2013weak}.
See~\cite{Joy} for a summary of these results.

Our first main result is a combinatorial description of $\mA_\mu(\pi)$, generalizing the description of $\mA_\mu(w_0)$ from~\cite{can2013weak}.
The definition of atoms for involutions can be extended to blocks, with $\mA(B)$ the set of block atoms for $B$ (see Section~\ref{ss:mu-atoms} for details).
The description is more easily stated in terms of inverse atoms, the set of which is $\mA_\mu^{-1}(\pi)$.

\begin{thm}[= Theorem~\ref{thm: mu involution atoms}]
\label{t:main-1}
    For $\pi$ a $\mu$--involution with blocks $B_1, \dots,B_k$,
    \begin{equation*}
        \mA_\mu^{-1}(\pi) = \{Q_1 Q_2 \dots Q_k: Q_i \in \mA^{-1}(B_i) \text{ for } i\in [k]\}.
    \end{equation*}
    where $Q_1Q_2 \dots Q_k$ is the concatenation of $Q_1,Q_2,\dots,Q_k$.
\end{thm}
Our proof of Theorem~\ref{t:main-1} adapts the proof from~\cite{can2013weak} to this more general setting.
As an application of this theorem, we expand $\mu$--involution Schubert polynomials as a multiplicity-free sum of $\nu$--involution Schubert polynomials when $\nu$ refines $\mu$ (see Theorem~\ref{t:mu-nu-expansion}).
This result generalizes Theorem~\ref{t:main-1} and allows for an improved description of Bruhat order for $\mu$--involutions when $\mu$ varies (see Corollary~\ref{c: intersect = contain}).

The other main result of this paper is a family of transition equations for $\mu$--involution Schubert polynomials, which is closely related to Monk's rule for Schubert polynomials.
The formula relies on the transposition operators $t_{ab}^\mu$ we define that describe the strong Bruhat cover relations of $\mu$--involutions.

\begin{thm}[= Theorem~\ref{t:mu-transition}]
\label{t:main-2}
    Let $\tau$ be a $\mu$ involution and $(i \geq j)$ a cycle in one of its blocks,
    \begin{align*}
        2^{- \delta(i,j)} (x_i+x_j)\fS_\tau^\mu =   \sum_{\sigma \in \Phi(\tau;i,j)}\fS_\sigma^\mu - \sum_{\sigma \in \Psi(\tau;i,j)}\fS_\sigma^\mu\, ,
    \end{align*}
    where 
    \begin{align*}
        \Phi(\tau;i,j) = \{t^\mu_{ab}(\tau) \neq \tau\, : \,  a \in \{i,j\}\},\quad \Psi(\tau;i,j) = \{t^\mu_{ab}(\tau) \neq \tau \, : \,  b \in \{i,j\}\}.
    \end{align*}
\end{thm}

A key technical tool in our proof of Theorem~\ref{t:main-2} is a novel exchange lemma like result for reduced expressions of $\mu$--involutions (see Theorem~\ref{t:mu-exchange}) generalizing an analogous result for involutions from~\cite{hamaker2018transition}.

\section*{Acknowledgements}
We thank Mahir Can, Yibo Gao, Allen Knutson, Eric Marberg, Alexander Postnikov, Linus Setiabrata, Jiayi Wen, Dora Woodruff, and Alexander Yong for helpful comments.
Both authors were partially supported by NSF Grant DMS-2054423.


\section{Background}

We first introduce some notation. Let $[n] = \{1,2,\dots,n\}$. For the rest of the paper, we will use $w,u,v$ to denote permutations, $y,z$ to denote involutions, $\pi, \tau, \sigma$ to denote $\mu$-involutions, and $\mu, \nu$ to denote compositions. Recall $\mu = (\mu_1,\dots,\mu_k)$ is a \newword{composition} of $n$ with \newword{length} $\ell(\mu) = k$, denoted $\mu \vDash n$, if $\sum_{i=1}^k \mu_i = n$ and $\mu_i \in [n]$ for all $i$.
For $\mu,\nu \vDash n$, we say $\mu$ \newword{refines} $\nu$ if
\[
\left\{\sum_{i=1}^j \mu_i: j \in [\ell(\mu)]\right\} \supseteq \left\{ \sum_{i=1}^j \nu_i : j \in [\ell(\nu)]\right\}.
\]
For example, $(3,2,3)$ refines $(5,3)$, $(3,5)$ and $(8)$.

\subsection{Permutations}
\label{ss:perm}
For $A \subseteq \NN$ finite, let $S_A$ be the set of permutations of $A$ and $S_n = S_{[n]}$.
We typically write permutations in one-line notation--- for $Q \in S_A$ write $Q = Q(1)\dots Q(k)$ with $Q(i) \in A$ distinct for each $i \in [k]$.
The \newword{standardization} of $Q$ is the permutation $\std(Q) \in S_{|A|}$ whose one-line notation has the same relative order as $Q$.
For $Q \in S_A$ and $a \in A$, let $Q^{-1}(a)$ be the index $i$ so that $Q(i) = a$.
Then the \newword{relative value} of $a$ in $Q$ is $r_Q(a) = \std(Q)(Q^{-1}(a))$.
For example, with $Q = 63274$ we have $\std(Q) = 42153$ so $Q^{-1}(4) = 5$ and $r_Q(4) = 3$.

When $n$ is clear, let $e = 12\dots n$ and $w_0 = n \dots 21$ be the identity and reverse permutations in $S_n$.
Let $t_{ij}$ denote the transposition $(i,j)$ and $s_i = (i,i{+}1)$.
Recall for $w \in S_n$ that $wt_{ij}$ is the same as $w$ except that $w(i)$ and $w(j)$ are swapped, while $t_{ij}w$ and $w$ are the same except the positions containing $i$ and $j$ are swapped.
For example, with $w = 4231$ we have $wt_{13} = 3241$ and $t_{13}w = 4213$.

Viewing permutations as functions, $S_n$ is a Coxeter group with presentation
\[
S_n = \langle s_1,\dots,s_{n-1} \mid s_i^2 = 1, s_i s_j = s_j s_i\ \mathrm{for}\ |i-j|>1,\ s_is_{i+1}s_i = s_{i+1}s_is_{i+1} \ \mathrm{for}\ i \in [n-2]\rangle.
\]
For $w \in S_n$, an expression $w = s_{a_1} \dots s_{a_p}$ is \newword{reduced} if it is of minimal length with associated \newword{reduced word} $\bfa = (a_1,\dots,a_p)$.
The \newword{length} $\ell(w)$ is then the length of a reduced word for $w$.
Let $\mR(w)$ be the set of reduced words for $w$.
Expressions for $w$ satisfy:
\begin{lem}[Deletion Property, \cite{bjorner2005combinatorics}*{Prop.~1.4.7}]
    \label{l:perm-exchange}
    For $w = s_{a_1} \dots s_{a_p} \in S_n$ so that the word $(a_1,\dots,a_p)$ is not reduced, there exists $k,m \in [p]$ so that $w  = s_{a_1} \dots \widehat{s_{a_k}} \dots \widehat{s_{a_m}} \dots  s_{a_p}$.
\end{lem}

We make extensive use of two partial orders on $S_n$.
The first is \newword{(strong) Bruhat order} $(S_n,\leq)$, where $v \leq w$ if every $\bfa \in \mR(w)$ has a subword in $\mR(v)$.
The second is \newword{weak (Bruhat) order} $(S_n,\leq_W)$ where $v \leq w$ if some $\bfa \in \mR(w)$ has a prefix in $\mR(v)$.
Both partial orders are graded with rank function $\ell$.
The cover relations for these partial orders are $v \lessdot w$ if $w = v t_{ij}$ and $\ell(w) = \ell(v) + 1$ and $v \lessdot_W w$ if $w = v s_i > w$ for some $i \in [n{-}1]$.

A \newword{parabolic subgroup} of $S_n$ is one generated by a subset of the simple transpositions.
In particular, for $\mu = (\mu_1,\dots,\mu_k) \vDash n$ we have the parabolic subgroup $S_\mu = S_{\mu_1} \times \dots \times S_{\mu_k}$; here we identify $S_{\mu_i}$ with $\langle s_{m_i + 1}, \dots, s_{m_i +\mu_{i}-1}\rangle$ where $m_i = \mu_1 + \dots + \mu_{i-1}$ (so $m_1 = 0)$.
The \newword{parabolic quotient} $S^\mu$ is the set of minimal length coset representatives for $S_\mu$.
We have:
\begin{prop}[{\cite{bjorner2005combinatorics}*{Prop.~2.4.4}}]
\label{p:parabolic}
    For $\mu \vDash n$ and $w \in S_n$ there exist unique $u \in S_\mu$, $v \in S^\mu$ so that $w = uv$.
    Furthermore $\ell(w) = \ell(u) + \ell(v)$.
\end{prop}

The \newword{0--Hecke monoid} is the unique monoid $(S_n,\circ)$ satisfying
\[
w \circ s_i = \begin{cases}
    w s_i & w s_i > w,\\
    w & ws_i < w.
\end{cases}
\]
This monoid is obtained from the Coxeter presentation of $S_n$ by replacing the relation $s_i^2 = 1$ with $s_i^2 = s_i$.
We say $\bfa = (a_1,\dots,a_p)$ is a \newword{word} for $w$ if $w = s_{a_1} \circ \dots \circ s_{a_p}$, and let $\mH_k(w)$ be the set of all words of length $k$ for $w$.
Note $\mR(w) = \mH_{\ell(w)}(w)$.

\subsection{Involutions}

Let $\mI_n = \{y \in S_n: y^{-1} = y\}$ be the set of involutions in $S_n$.
For $y \in \mI_n$, the \newword{cycles} of $y$ are elements in the set of ordered pairs
\begin{align*}
    \Cyc(y) = \{(j,i) \in [n] \times [n]\, |\, y(i) = j\, \text{ and }i \leq j\}.
\end{align*}
Note that fixed points correspond to terms $(i,i) \in \Cyc(y)$.
We say the cycles $(i,j),(k,l) \in \Cyc(y)$ \newword{cross} if $j < l < i < k$ or $l < j < i < k$ and \newword{nest} if $j < l \leq k <i$ or $l < j \leq i < k$.

For $w \in S_n$ and $y \in \mI_n$, the 0--Hecke monoid acts on $\mI_n$ by $y \iprod w = w^{-1} \circ y \circ w$~\cite{richardson1990bruhat}.
Bruhat order restricts to a graded poset $(\mI_n,\leq^\mI)$ with rank function 
\[
\ellhat(y) = \frac{1}{2}(\ell(y)+n - |\Cyc(y)|).
\]
The \newword{weak order for involutions} $(\mI_n,\leq^\mI_W)$ is the partial order with $y 
\leq^\mI_W z$ if $z = y \iprod w$ for some $w \in S_n$.
In~\cite{incitti2004bruhat}, Incitti characterized the cover relations of $\leq^\mI$ combinatorially in terms of operators $t^\mI_{ij}$, which mimic the action of transpositions on permutations.
For $y \in \mI_n$ with $(b,a),(d,c) \in \Cyc(y)$ and $i,j \in \{a,b,c,d\}$, let $t^\mI_{ij}(y)$ differ from $y$ on the set $\{a,b,c,d\}$ by modifying $\Cyc(y)$ at these values as depicted in Figure~\ref{fig:I-bruhat-cover} and fixing $y$ otherwise.
Then $y \lessdot^\mI z$ if and only if $\ellhat(z) = \ellhat(y) +1$ and $z = t^\mI_{ij}(y)$ for some $i<j$.
Note our conventions follow~\cite{hamaker2018transition} and  differ slightly from~\cite{incitti2004bruhat}.

\begin{figure}
    \centering
    \[
\ba
t^\mI_{ij} \Bigl(\arcstart
{
*{.}     & *{.}
}
\arcstop \Bigr)
&=
\arcstart
{
*{.}   \arc{.6}{r}  & *{.}
}
\arcstop
&&\text{for }(i,j) =(a,c)
\\
t^\mI_{ij}  \Bigl( \arcstart
{
*{.}    \arc{.6}{r}  & *{.} & *{.}
}
\arcstop\Bigr)
&=
 \arcstart
{
*{.}    \arc{.8}{rr}  & *{.} & *{.}
}
\arcstop
&&\text{for }(i,j) \in \{ (b,c), (a,c)\}
\\
t^\mI_{ij}  \Bigl( \arcstart
{
*{.}     & *{.}  \arc{.6}{r} & *{.}
}
\arcstop\Bigr)
&=
 \arcstart
{
*{.}    \arc{.8}{rr}  & *{.} & *{.}
}
\arcstop
&&\text{for }(i,j) \in \{ (a,d), (a,c)\}
\\
t^\mI_{ij}  \Bigl( \arcstart
{
*{.}  \arc{.6}{r}   & *{.}    & *{.} \arc{.6}{r}& *{.}
}
\arcstop\Bigr)
&=
\arcstart
{
*{.}  \arc{.8}{rr}   & *{.} \arc{.8}{rr}   & *{.} & *{.}
}
\arcstop
&&\text{for }(i,j) =(b,c)
\\
t^\mI_{ij}  \Bigl( \arcstart
{
*{.}  \arc{.6}{r}   & *{.}    & *{.} \arc{.6}{r}& *{.}
}
\arcstop\Bigr)
&=
\arcstart
{
*{.}  \arc{.8}{rrr}   & *{.}    & *{.} & *{.}
}
\arcstop
&&\text{for }(i,j) \in \{ (a,c),(b,d),(a,d)\}
\\
t^\mI_{ij} \Bigl( \arcstart
{
*{.}  \arc{.8}{rr}   & *{.} \arc{.8}{rr}   & *{.} & *{.}
}
\arcstop\Bigr)
&=
\arcstart
{
*{.}  \arc{.8}{rrr}   & *{.}  \arc{.4}{r}  & *{.} & *{.}
}
\arcstop
&&\text{for }(i,j) \in \{ (a,d), (b,d), (a,d)\}
\ea
\]
    \caption{The action of $t^\mI_{ij}$ on $(b,a),(d,c) \in \Cyc(y)$ with $a < c$.
    When $a = b$ or $c=d$, we only use the earlier letter.}
    \label{fig:I-bruhat-cover}
\end{figure}

An \newword{involution word} for $y \in \mI_n$ is $\bfa = (a_1,\dots,a_p)$ so that $y = s_{a_1} \iprod s_{a_2} \iprod \dots \iprod s_{a_p}$ (note the action is not a product, so must be performed from left to right).
An involution word is \newword{reduced} if it is minimal length, or equivalently if $p = \ellhat(y)$.
Let $\hat{\mR}(y)$ be the set of reduced involution words.
The \newword{atoms} for $y \in \mI_n$ are the permutations in the set
\[
\mA(y) = \{w \in S_n: w^{-1} \circ w = y,\ \ell(w) = \ellhat(y)\}.
\]
Since $\iprod$ is a 0--Hecke monoid action, we see for $\bfa \in \hat{\mR}(y)$ that $\bfa \in \mR(w)$ for some $w \in \mA(y)$, and furthermore that 
\[
\hat{\mR}(y) = \bigcup_{w \in \mA(y)} \mR(w).
\]
There is a surprisingly simple combinatorial description of atoms.

\begin{thm}[{\cite{HMP2}*{Cor 5.13}; see also \cite{CJW}*{Thm 2.5}}]
    \label{thm: involution atoms}
    For $y \in \mathcal{I}_n$, $w \in \mathcal{A}(y)$ iff:
    \begin{enumerate}
        \item If $(i,j) \in \mathrm{Cyc}(y)$, then $w(i) \geq w(j)$.
        \item If $(i,j) \in \mathrm{Cyc}(y)$, then there does not exist $i < k < j$ such that $w(i) > w(k) > w(j)$.
        \item If $(i,j), (k,l) \in \Cyc(y)$ with $i < k$ and $j < l$, then 
        $w(k) \geq w(l) > w(i) \geq w(j)$.
    \end{enumerate}
\end{thm}
For $y \in \mI_n$, Theorem~\ref{thm: involution atoms} is perhaps more easily understood in terms of the \newword{inverse atoms} $\mathcal{A}^{-1}(y) = \{w^{-1}: w \in \mA(y)\}$.
Here, for $v \in \mA^{-1}(y)$ and $(i,j),(k,l) \in \Cyc(y)$ with $i > j$,
\begin{itemize}
    \item (1) says that $i$ occurs before $j$ in $v$;
    \item from (2), when $(i,j)$ and $(k,l)$ nest with $j < l \leq k < i$, neither $k$ nor $l$ occurs between $i$ and $j$ in $v$---the possible subwords of $v$ they form are $ijkl$, $kijl$ and $klij$;
    \item from (2) and (3), when $(i,j)$ and $(k,l)$ cross with $j < l < i < k$, then $l$ occurs after $i$ (hence also $j$), so they form the subword $ijkl$;
    \item when $(i,j)$ and $(k,l)$ neither nest or cross, say with $i < l$, from (3) $i,j$ both occur before $k,l$ in $v$.
\end{itemize}

There is natural graph structure on inverse atoms encoded in the following theorem.
\begin{thm}[\cite{HMP2}*{Thm.~6.10}]
    \label{t: atom local move}
    For $y$ an involution, view $\mA^{-1}(y)$ as a graph with edges
    \begin{equation}
        \label{eq:atom-transformation}
    u = [\cdots c,a,b \cdots] \sim  v = [\cdots b,c,a \cdots]
    \end{equation}
    for consecutive subsequences where $a<b<c$ and other entries are unchanged.
    Then this graph is connected.
\end{thm}

\begin{exa}
    For $y = 5472163$, we have $\Cyc(y) = \{(5,1),(4,2),(7,3),(6,6)\}$ and
\[
\mA^{-1}(y) = \{4251673,4512673,5142673,4251736,4512736,5142736\}
\]
by Theorem~\ref{thm: involution atoms}.
For any cycle $(i, j) \in \Cyc(y)$, the $i$ must appear to left of $j$ in any inverse atom. By part (2) of the theorem, since $(4,2)$ is nested in $(5,1)$, the numbers $\{4,2\}$ must not appear between $\{5,1\}$, similarly for $\{6\}$ and $\{7,3\}$. By part (3) of the theorem, the numbers $\{5,1,4,2\}$ must appear before the numbers $\{7,3,6\}$. By Theorem~\ref{t: atom local move}, $\mA^{-1}(y)$ is connected by local moves, and the graph induced on $\mA^{-1}(y)$ by~\eqref{eq:atom-transformation} is
\[
\begin{tikzpicture}
    \node at (0,0) (a) {4251673};
    \node at (3,0) (b) {4512673};
    \node at (6,0) (c) {5142673};
    \node at (0,1) (d) {4251736};
    \node at (3,1) (e) {4512736};
    \node at (6,1) (f) {5142736};
    \draw (a) -- (b) -- (c) -- (f) -- (e) -- (d) -- (a);
    \draw (b) -- (e);
\end{tikzpicture}
\]
with vertical edges corresponding to $673 \leftrightarrow 736$ and horizontal arrows corresponding to transformations where $a = 1$, $c = 5$, and $b= 4$ or $2$.
\end{exa}

For $A = \{a_1 < \dots < a_k\} \subseteq [n]$ and $w \in S_n$, let $[w]_A = w(a_1) \dots w(a_k) \in S_{w(A)}$.
The following is a slight extension of Theorem 1.3 in~\cite{hamaker2018transition}.

\begin{theorem}
    \label{t:atom-transposition}
    Let $y,z \in \mI_n$ with $y \lessdot^\mI z$. Then
    \[
    \mA(z) = \{v t_{ij}: v \in \mA(y), t^\mI_{ij}(y) = z, v \lessdot vt_{ij}\}.
    \]
\end{theorem}

\begin{proof}
    The containment $\supseteq$ is an immediate corollary of~\cite{hamaker2022involution}*{Thm.~4.22}, which is itself a restatement of Theorem 1.3 and related results from~\cite{hamaker2018transition}.
    For the opposite containment $\subseteq$, let $w \in \mA(z)$ and $i<j$ so that $t^\mI_{ij}(y) = z$.
    Define $E = \{i,j,y(i),y(j)\}$ and note $E = \{i,j,z(i),z(j)\}$ by the definition of $t^\mI_{ij}$.
    From Corollary 3.19 (b) in~\cite{hamaker2018transition}, $[w]_E \in \mA([z]_E)$, and direct inspection reveals there exists a unique transposition $t = (i',j')$ with $i',j' \in E$ so that $[wt]_E \in \mA([y]_E)$.
    Then $t^\mI_{i'j'}(y) = z$, so by~\cite{hamaker2022involution}*{Thm.~4.22} we see $wt \in \mA(y)$ and the result follows.
\end{proof}

We can state a property for atoms roughly equivalent to the Exchange Property for involution words.

\begin{thm}[{\cite{hamaker2018transition}*{Thm 3.23}}]
\label{t:atoms-exchange}
    Fix $y \in \mI_n$, $v \in \mA(y)$ and $w$ so that $v \lessdot w$ and $w$ is not an atom for any $z \in \mI_n$.
    Then there exists unique  $u \neq v$ with $u \lessdot w$ and $u  \in \mA(y)$.
    
\end{thm}


\subsection{$\mu$--involutions}
\label{ss:mu-involutions}

For $A \subseteq \NN$ with $|A| = k$, $Q \in S_A$ is a \newword{block involution} if $\std(Q) \in \mI_k$.
Then the \newword{block cycles} of $Q$ are elements of the set
\[
\Cyc(Q) = \{(Q(j),Q(i)): (j,i) \in \Cyc(\std(Q))\}.
\]
For example,  $Q = 927681$ is a block involution since $\std(Q) = 624351 \in \mI_6$, and the block cycles of $Q$ are $\Cyc(Q) = \{(9,1),(2,2),(7,6),(8,8)\}$.

\begin{defn}
\label{def:mu-involution}
For $\mu  = (\mu_1, \cdots, \mu_k)$ a composition of $n$ and $\pi \in S_n$, the \newword{block decomposition} of $(\pi,\mu)$ is the tuple of words $(\pi^{(1)},\dots,\pi^{(k)})$ where for $m_i = \mu_1 + \dots + \mu_{i-1}$
\[
B_i = (\pi(m_i{+}1) ,\dots,\pi(m_i+\mu_i)).
\]
Note here that $m_1 = 0$ and $m_{k+1} = n$.
Call $B_i$ the $i$th \newword{block} in $(\pi,\mu)$.
Then $\pi$ is a \newword{$\mu$--involution} if each block in $(\pi,\mu)$ is a block involution.
Let 
\[
\mI_\mu = \{\pi \in S_n: (\pi,\mu) \text{ is a } \mu\text{-involution}\}.
\]
Note $\mI_{(n)} = \mI_n$ and $\mI_{(1^n)} = S_n$, also that $e,w_0 \in \mI_\mu$ for all $\mu$.
\end{defn}

We will frequently depict $\mu$--involutions by separating the blocks with $|$'s.
For example, with $\pi = 5163742$ and $\mu = (3,1,3)$, to interpret $\pi$ as a $\mu$--involution we write $\pi = 516|3|742$.
To verify $\pi$ is a $\mu$--involution, we check $B_1$, $B_2$ and $B_3$ are all block involutions:
\[\std(B_1) = \std(516) = 213, \quad \std(B_2) = \std(3) =1, \quad\mathrm{and}\quad \std(B_3) = \std(742) = 321.
\]

For $\mu = (\mu_1,\dots,\mu_k)$, the \newword{cycles} of a $\mu$--involution $\pi$ are elements of
\[
\Cyc_\mu(\pi) = \bigsqcup_{i = 1}^k \Cyc(B_i).
\]
When writing $\Cyc_\mu(\pi)$, we order the cycles first by the order of the blocks, then within each block in increasing order based on the smallest number.
We call this the \newword{$\mu$--cycle order}.
For example, with $(\pi,\mu)$ as above
\[
\Cyc_\mu(\pi) = \{(5,1),(6,6),(3,3),(7,2),(4,4)\}.
\]
There is a unique $y_\pi \in \mI_n$ so that $\Cyc_\mu(\pi) = \Cyc(y_\pi)$.
We can then interpret $\pi$ as an ordered set partition $O(\pi) = (S_1,\dots,S_k)$ of $\Cyc(y_\pi)$.
Continuing our previous example, we see $y_\pi = 5734162$ as
\[
O(\pi) = (5,1),(6,6)|(3,3)|(7,2),(4,4)
\]
is an ordered set partition of $\Cyc(5734162)$.

The 0--Hecke monoid acts on $\mI_\mu$.
For $\pi \in \mI_\mu$, define
\[
\pi \circ_\mu s_i = \begin{cases}
    s_i \circ \pi & i,i{+1} \text{ in different blocks,}\\
    s_i \circ \pi \circ s_{r_{B}(i)} &  i,i{+1} \text{ in the same block } B,
\end{cases}
\]
where $r_B(i)$ denotes the relative value of $i$ in the block $B$ as defined in Section~\ref{ss:perm}.
In the first case, we swap the values $i$ and $i{+1}$ in $\pi$ between blocks if $i$ occurs before ${i+}1$.
In the latter case, we act on the standardization within the block involution $B$ containing both $i$ and $i{+1}$ via $\std(B) \iprod s$, where $s$ corresponds to the relative value of $i$ in its block.
For example, with $\pi = 516|3|742$ (so $\mu = (3,1,3)$) we have
\[
\pi \circ_\mu s_4 = \pi, \quad \pi \circ_\mu s_3 = 516|4|732, \quad \pi \circ s_5 = 651 |3|742,
\]
with the last product corresponding to $213 \iprod s_2 = 321$.

The operation $\muprod$ extends to an action of the 0--Hecke monoid.

\begin{prop}[\cite{can2013weak}*{Lem.~4.4}]
\label{p:mu-monoid}
For $\mu \vDash n$, the operation $\muprod$ extends to an action of $(S_n,\circ)$.
\end{prop}

We say $\bfa = (a_1,\dots,a_p)$ is a \newword{$\mu$--word} for $\pi \in \mI_\mu$ if
\[
\pi = e \muprod s_{a_1} \muprod \dots \muprod s_{a_p}.
\]
A $\mu$--word is \newword{reduced} if it is of minimal length.
Let $\ellmu(\pi)$ be this minimal length and $\mR_\mu(\pi)$ be the set of reduced $\mu$--words for $\pi$.
By Proposition~\ref{p:mu-monoid}, $\ellmu$ is well defined and
\[
\mR_\mu(\pi) = \bigsqcup_{w \in \mA_\mu(\pi)} \mR(w)
\]
for some set $\mA_\mu(\pi) \subseteq \{w \in S_n: \ell(w) = \ellmu(\pi)\}$.
Each operation $\muprod s_i$ acts either within a block or between blocks.
Therefore, we can decompose
\[
\ellmu(\pi) = \ellmu^\mI(\pi) + \ellmu^{B}(\pi)
\]
where $\ellmu^\mI(\pi)$ is the number of $\muprod s_i$'s acting within blocks and $\ellmu^B(\pi)$ the number acting between blocks.
We give a combinatorial characterization of $\mA_\mu(\pi)$ in Theorem~\ref{thm: mu involution atoms}.

For $\mu \vDash n$, the \newword{weak $\mu$--Bruhat order} on $\mI_\mu$ is the partial order $\leq^\mu_W$ where $\pi \leq^\mu_W \tau$ if $\tau = \pi \muprod w$ for some $w \in S_n$.
As mentioned in the introduction, $\mu$--involutions correspond to Borel orbits in the variety of complete quadrics, or equivalently $H_\mu$--orbits in the flag variety.
The \newword{$\mu$--Bruhat order} on $\bigsqcup_{\mu \vDash n} \mI_\mu$ is the partial order $\leq^\mu$ given by orbit closure containment.
Note when comparing $\pi$ a $\mu$--involution and $\tau$ a $\nu$--involution that we still use $\leq^{\mu}$.
For weak $\mu$--Bruhat order, such $\pi$ and $\tau$ cannot be compared.

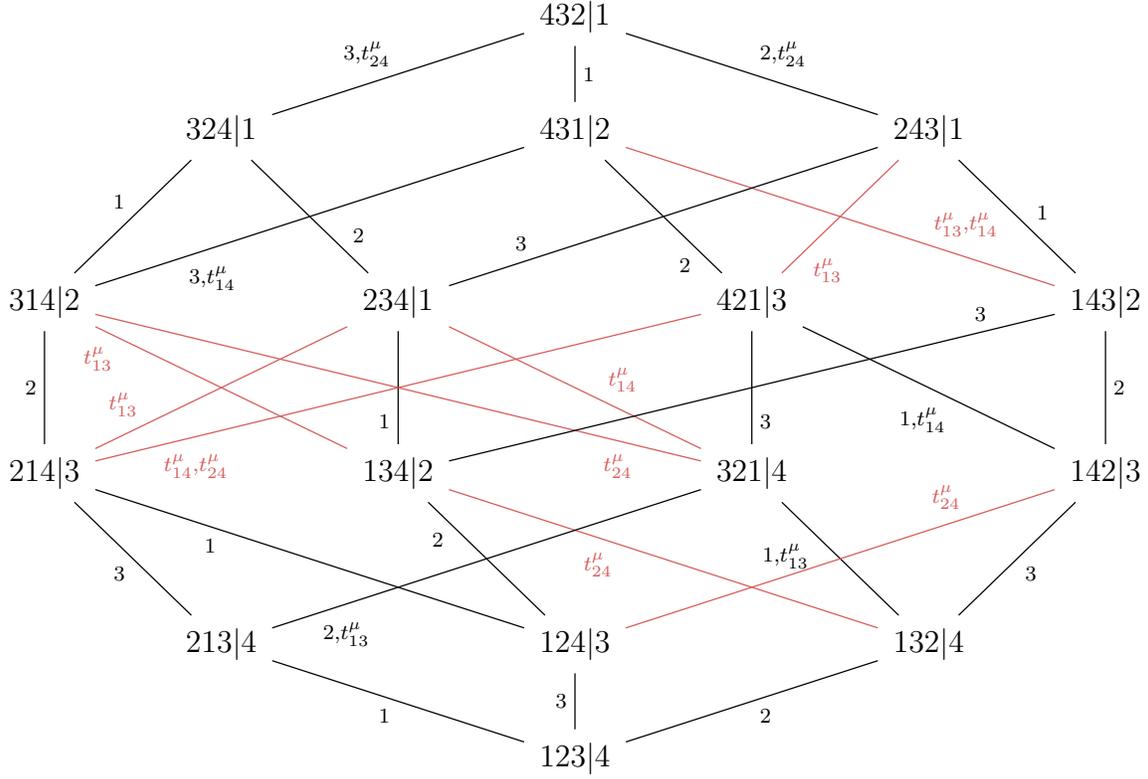
\begin{figure}
    \centering
\[\begin{tikzcd}
	&&& {432|1} &&& \\
	& {324|1} && {431|2} && {243|1} \\
	\\
	{314|2} && {234|1} && {421|3} && {143|2} \\
	\\
	{214|3} && {134|2} && {321|4} && {142|3} \\
	\\
	& {213|4} && {124|3} && {132|4} \\
	&&& {123|4}
	\arrow["{3, t_{24}^\mu}"', no head, from=1-4, to=2-2]
	\arrow["1", no head, from=1-4, to=2-4]
	\arrow["{2, t_{24}^\mu}", no head, from=1-4, to=2-6]
	\arrow["1"', no head, from=2-2, to=4-1]
	\arrow["2"{pos=0.8}, no head, from=2-2, to=4-3]
	\arrow["2"'{pos=0.8}, no head, from=2-4, to=4-5]
	\arrow["{t_{13}^\mu,t_{14}^\mu}"{pos=0.7}, color={rgb,255:red,214;green,92;blue,92}, no head, from=2-4, to=4-7]
	\arrow["3"'{pos=0.8}, no head, from=2-6, to=4-3]
	\arrow["{t_{13}^\mu}"{pos=0.8}, color={rgb,255:red,214;green,92;blue,92}, no head, from=2-6, to=4-5]
	\arrow["1"{pos=0.6}, no head, from=2-6, to=4-7]
	\arrow["{3,t_{14}^\mu}"'{pos=0.2}, no head, from=4-1, to=2-4]
	\arrow["2"', no head, from=4-1, to=6-1]
	\arrow["{t_{24}^\mu}"'{pos=0.9}, color={rgb,255:red,214;green,92;blue,92}, no head, from=4-1, to=6-5]
	\arrow["1"'{pos=0.8}, no head, from=4-3, to=6-3]
	\arrow["{t_{14}^\mu}"{pos=0.6}, color={rgb,255:red,214;green,92;blue,92}, no head, from=4-3, to=6-5]
	\arrow["{t_{14}^\mu, t_{24}^\mu}"{pos=0.9}, color={rgb,255:red,214;green,92;blue,92}, no head, from=4-5, to=6-1]
	\arrow["3"{pos=0.8}, no head, from=4-5, to=6-5]
	\arrow["{1, t_{14}^\mu}"'{pos=0.6}, no head, from=4-5, to=6-7]
	\arrow["3"'{pos=0.1}, no head, from=4-7, to=6-3]
	\arrow["2", no head, from=4-7, to=6-7]
	\arrow["{t_{13}^\mu}"{pos=0.2}, color={rgb,255:red,214;green,92;blue,92}, no head, from=6-1, to=4-3]
	\arrow["3"', no head, from=6-1, to=8-2]
	\arrow["1"'{pos=0.3}, no head, from=6-1, to=8-4]
	\arrow["{t_{13}^\mu}"{pos=0.9}, color={rgb,255:red,214;green,92;blue,92}, no head, from=6-3, to=4-1]
	\arrow["2"'{pos=0.2}, no head, from=6-3, to=8-4]
	\arrow["{t_{24}^\mu}"'{pos=0.4}, color={rgb,255:red,214;green,92;blue,92}, no head, from=6-3, to=8-6]
	\arrow["{2, t_{13}^\mu}"{pos=0.9}, no head, from=6-5, to=8-2]
	\arrow["{1, t_{13}^\mu}"'{pos=0.3}, no head, from=6-5, to=8-6]
	\arrow["{t_{24}^\mu}"'{pos=0.2}, color={rgb,255:red,214;green,92;blue,92}, no head, from=6-7, to=8-4]
	\arrow["3", no head, from=6-7, to=8-6]
	\arrow["1"', no head, from=8-2, to=9-4]
	\arrow["3"', no head, from=8-4, to=9-4]
	\arrow["2", no head, from=8-6, to=9-4]
\end{tikzcd}\]
    \caption{The black edges form the weak $\mu$-Bruhat order for $\mu = (3,1)$.
    Their edge labels correspond
 to the appropriate $s_i$ for the corresponding $0$--Hecke monoid action.
    The \textcolor{rgb,255:red,214;green,92;blue,92}{red edges} are $\mu$-Bruhat order covers that are not weak $\mu$-Bruhat covers.
    They are labeled with transposition-like operators $t^\mu_{ij}$ witnessing these covers that are defined in Section~\ref{ss:transposition}.}
    \label{fig: bruhat order}
\end{figure}

In much of our work we restrict our attention to the order $(\mI_\mu,\leq^\mu)$, which strengthens weak $\mu$--Bruhat order in the sense that $\pi \leq^\mu_W \tau$ implies $\pi \leq^\mu \tau$~\cite{timashev1994generalization}*{\S2.9~(1)}.
Furthermore, $\leq^\mu$ has a combinatorial description via subword containment that is nearly equivalent to~\cite{banerjee2016combinatorial}*{Thm.~1.1~(ii)}.
For our purposes, this theorem can be taken as the definition of $\mu$--Bruhat order.

\begin{theorem}
    \label{t:mu-bruhat-def}
    For $\mu \vDash n$ and $\pi, \tau \in \mI_\mu$, we have $\pi \leq^\mu \tau$ if and only if each $\bfa \in \mR_\mu(\tau)$ has a subword $\bfa' \in \mR_\mu(\pi)$.
\end{theorem}

\begin{proof}
    Summarizing~\cite{timashev1994generalization}*{\S2.9~(5)} in the $\mu$--involution setting,~\cite{banerjee2016combinatorial}*{Prop.~5.1} states $\pi \leq^\mu \tau$ if and only if $\pi = \tau$ or there exist $\pi',\tau' \in \mI_\mu$ and $i \in [n-1]$ so that $\pi' \muprod s_i = \pi \neq \pi'$, $\tau' \muprod s_i = \tau$ and $\pi' \leq^\mu \tau'$.
    Call this the `one step property'.

    With $\mu,\tau \in \mI_\mu$, let $\bfa = (a_1,\dots,a_p) \in \mR_\mu(\tau)$ and $\tau' = s_{a_1} \muprod \dots \muprod s_{a_{p-1}}$, so $\tau' \leq^\mu_W \tau$ hence $\tau' \leq^\mu \tau$ as well.
    If $\pi \muprod s_{a_p} \neq \pi$, we see by the one step property that $\pi \leq^\mu \tau'$ if and only if $\pi \muprod s_{a_p} \leq^\mu \tau$.
    Since $\pi \leq \pi \muprod s_{a_p}$, we can reduce to the case where $\pi \muprod s_{a_p} = \pi$.
    In this case, there exists $\pi'$ with $\pi' \muprod s_{a_p} = \pi$.

    We argue by induction first on $\ellmu(\tau)$, second on $\ellmu(\tau) - \ellmu(\pi)$ with base case $\tau =e$ where both quantities are 0.
    By the one step property $\pi \leq^\mu \tau$ if and only if $\pi' \leq^\mu\tau'$.
    By the induction hypothesis, $\pi' \leq^\mu\tau'$ if and only if $(a_1,\dots,a_{p-1}) \in \mR_\mu(\tau')$ has a subword in $\mR_\mu(\pi')$.
    Appending $a_p$ to this word gives $\bfa' \in \mR_\mu(\pi)$ a subword of $\bfa$, so $\pi \leq^\mu \tau$ if and only if each $\bfa \in \mR_\mu(\tau)$ contains as a subword some $\bfa' \in \mR_\mu(\pi)$.
\end{proof}

In Section~\ref{ss:transposition} we present a combinatorial description of the cover relations for $(\mI_\mu,\leq^\mu)$.
This will allow us to extend Theorem~\ref{t:atom-transposition} to $\mu$--involutions.
The remaining cover relations in $(\bigsqcup_{\mu \vDash n} \mI_\mu,\leq^\mu)$ are characterized by:

\begin{theorem}[\cite{banerjee2016combinatorial}*{Thm.~1.1~(i)}]
\label{t:mu-nu-cover}
    For $\mu,\nu \vDash n$ with $\pi \in \mI_\mu$, $\tau \in \mI_\nu$, we have $\tau \lessdot^\mu \pi$ if and only if $\nu$ refines $\mu$ with $\ell(\mu) = \ell(\nu)+1$ and $\mA_\mu(\tau) \subseteq \mA_\nu(\pi)$.
\end{theorem}

Motivated by this characterization, for $\pi \in \mI_\mu$ and $\tau \in \mI_\nu$ we say $\tau \leq_\mA \pi$ if $\nu$ refines $\mu$ and $\mA_\nu(\tau) \subseteq \mA_\mu(\pi)$.
Equivalently $\tau \leq_\mA \pi$ if $\ell_\nu(\tau) = \ell_\mu(\pi)$ and $\tau \leq^\mu \pi$.
We give a more explicit description of these relations in Section~\ref{ss:mu-to-nu}.

\subsection{Schubert polynomials and transition equations}

    Schubert polynomials, introduced by Lascoux and Sch\"{u}tzenberger, represent the cohomology classes of Schubert varieties in the flag variety~\cites{ls82a}. They are a family of polynomials indexed by permutations.
    For $w \in S_n$, the Schubert polynomial of $w$, denoted $\fS_w$, can be defined recursively using divided difference operators.
    
    Let $S_n$ act on the polynomial ring $\ZZ[x_1, \dots, x_n]$ by permuting variable subscripts.
    For $i \in [n]$, the \definition{divided difference operator $\partial_i$} is defined by
        \begin{align*}
            \partial_i(f) := \frac{f - s_i(f)}{x_i - x_{i+1}}.
        \end{align*}
    Note that $\partial_i(f) \in \ZZ[x_1, \dots, x_n]$ and $s_i \circ \partial_i(f) = \partial_i(f)$, that is $\partial_i(f)$ is symmetric in the variables $x_i$ and $x_{i+1}$.
    The \definition{Schubert polynomial $\mathfrak{S}_w$} is defined by
        \begin{align*}
            \fS_w :=
            \begin{cases} 
              x_1^{n-1}x_2^{n-2}\cdots x_{n-2}^2x_{n-1}^1 & w = w_0, \\
              \partial_i(\fS_{w s_i}) & w(i) < w(i+1).
           \end{cases}
         \end{align*}

    \begin{exa}
        We compute $\mathfrak{S}_{312}$ using the divided difference operators
         \begin{align*}
             \fS_{312} &=  \partial_2(\fS_{312 \cdot s_2}) = \partial_2(\fS_{321})= \partial_2(x_1^2x_2) \\
             &= \frac{x_1^2x_2 - x_1^2x_3}{x_2-x_3} = \frac{x_1^2(x_2 - x_3)}{x_2-x_3} = x_1^2\,.
         \end{align*}
     \end{exa}

   Computing the product of Schubert polynomials combinatorially is an important open problem.
   This product structure is governed by Monk's rule, which we now state.
    \begin{thm}[\cite{monk1959}*{Thm.~3}]
    For $w \in S_n$ and a positive integer $r$, 
        \begin{align*}
            \fS_{s_r}\fS_w = (x_1 + \cdots + x_r)\fS_w = \sum_{\substack{k \leq r < l\\ w \lessdot wt_{k,l}}} \fS_{wt_{k,l}}\,.
        \end{align*}
    \end{thm}
Transition equations are similar recurrences on Schubert polynomials.
They are direct consequences of Monk's rule, and are often more useful in computation.
\begin{thm}
\label{t:monk}
    For $v \in S_n$,
    \begin{align*}
        x_r \fS_v = \sum_{\substack{ r < s \\   v \lessdot vt_{r,s}}}  \fS_{vt_{r,s}} - \sum_{\substack{h <r \\ v \lessdot (vt_{h,r})}} \fS_{vt_{h,r}}.
    \end{align*}
\end{thm}

\begin{rem}
    Some authors use transition to refer exclusively to the case where $v = w s_r$ with $r$ the largest (right) descent in $w$.
    We do not adopt this convention.
\end{rem}

In~\cite{wyser2017polynomials}, Wyser and Yong introduced polynomial representatives $\Upsilon_{y}^{\mathrm{O}_n}$ indexed by $y \in \mI_n$ for cohomology classes of orthogonal group orbit closures in $\text{Fl}_n$.
For $y \in \mI_n$ the \newword{involution Schubert polynomial}, originally defined in terms of divided differences, is
\[
\fS_y^{\mI} = \sum_{w \in \mA(y)} \fS_w,
\]
following work of Brion~\cite{brion1998behaviour}.
Involution Schubert polynomials are rescaled versions of Wyser and Yong's polynomial representatives, satisfying $\Upsilon_y^{\mathrm{O}_n} = 2^{\kappa(y)}\fS_y^{\mI}$, where $\kappa(y)$ is the number of two cycles in $y$.
We work with $\fS_y^\mI$ rather than $\Upsilon^{\mathrm{O}_n}_Y$ to simplify our computations, as the power of 2 can easily be added back to recover the geometric significance.

Involution Schubert polynomials satisfy an analogue of the transition equation.
    \begin{thm}[\cite{hamaker2018transition}*{Thm.~1.5}]
    \label{t:inv-monk}
        Let $y \in \mI_\infty$ and $(p,q) \in \Cyc(y)$ then
        \begin{align*}
            2^{-\delta(p,q)}(x_p + x_q)\fS_y^\mI = \sum_{z \in \Phi^+(y,q)} \fS_z^\mI - \sum_{z \in \Phi^-(y,p)} \fS_z^\mI
        \end{align*}
        where 
        \begin{align*}
            \Phi^{+}(y,q) = \{\, z \in \mI_\ZZ : y \lessdot_{I} z \text{ and } z = t_{qj}^\mI(y) \text{ for an integer } j > q \,\}\\ \Phi^{-}(y,p) = \{\, z \in \mI_\ZZ : y \lessdot_{I} z \text{ and } z = t_{ip}^\mI(y) \text{ for an integer } i < p \,\}
        \end{align*}
    \end{thm}

   As mentioned in the introduction, there is a parallel story for the cohomology representatives of $H_\mu$-orbit closures in $\text{Fl}_n$.
    Up to rescaling, these representatives are the \definition{$\mu$--involution Schubert polynomials}, defined for $\tau \in \mI_\mu$ as 
        \begin{align*}
            \fS_\tau^\mu = \sum_{w \in \mA_\mu(\tau)} \fS_w.
         \end{align*}
    To recover the geometric meaning, one multiplies $\fS^\mu_\tau$ by $2^\kappa(\tau)$ where $\kappa(\tau)$ is the number of 2--cycles in $\tau$, which are distinct pairs in $\Cyc_\mu(\tau)$.
    Note $\fS^\mu_\pi$ can also be defined in terms of divided difference operators~\cite{can2015wonderful}, but we do not make use of this construction.

\section{Coxeter combinatorics for $\mu$--involutions}
\label{s:mu-involution-combinatorics}

\subsection{Atoms of $\mu$--involutions}
\label{ss:mu-atoms}
For $\mu$ a composition, we now characterize the $\mu$--atoms of $\pi \in \mI_\mu$ in terms of their inverses. Recall that $\mA(y)$ and $\mA^{-1}(y)$ denote the set of atoms and inverse atoms of an involution $y$ respectively. These definitions can be naturally extended to block involutions. Let $A \subset \NN$ be finite and $B \in S_A$ be a block involution.
The \newword{block atoms} for $B$ are elements in the set $
\mA(B) = \{Q \in S_A: \std(Q) \in \mA(\std(B)\},
$
that is the permutations in $S_A$ whose standardization is an atom for $\std(B)$.
Similarly, the \newword{inverse block atoms} for $B$ are elements in the set $\mathcal{A}^{-1}(B) = \{Q \in S_A: \std(Q) \in \mA^{-1}(\std(B))\}$, that is the permutations in $S_A$ whose standardization is an inverse atom for $\std(B)$.

\begin{thm}
\label{thm: mu involution atoms}
    Let $\mu \vDash n$ and $\pi \in \mathcal{I}_{\mu}$ with blocks $B_1, \dots,B_k$.
    Then
    \[
        \mA_\mu^{-1}(\pi) = \{Q_1 Q_2 \dots Q_k: Q_i \in \mA^{-1}(B_i) \text{ for } i\in [k]\}.
    \]
    where $Q_1Q_2 \dots Q_k$ is the concatenation of $Q_1,Q_2,\dots,Q_k$.
\end{thm}

\begin{proof}
    For $w \in \mA^{-1}_\mu(\pi)$, write $w = uv$ with $u \in S_\mu$ and $v \in S^\mu$ using Proposition~\ref{p:parabolic}.
    Then $w^{-1} = v^{-1} u^{-1} \in \mA_\mu(\pi)$ with $u^{-1} \in S_\mu$.
    Necessarily, $e \muprod u$ cannot swap values between blocks, so $u^{-1} \in \mA_\mu(\tau)$ where $\tau$ has blocks $B'_1, \dots, B'_k$.
    Since $\muprod$ acts on the left when swapping values between blocks, we see $\tau \muprod v^{-1}$ acts by swapping values between blocks in the same way as $v$.
    Therefore $\ell(v) = \ell^B_\mu(\pi)$, hence $\ell(u) = \ell^\mI_\mu(\pi)$.
    Since $w^{-1}$ is a $\mu$--atom  for $\pi$ and $v$ only swaps values between blocks, necessarily $\std(B_i) = \std(B_i')$ for all $i \in [k]$.
    Then $u$ splits into inverse block atoms $Q_1',\dots,Q_k'$ with $Q_i' \in \mA^{-1}(B_i')$ and $v$ permutes the values of each $Q'_i$ to produce $Q_i \in \mA^{-1}(B_i)$.
    A permutation of the form $Q_1 \dots Q_k$ with each $Q_i \in \mA(B_i)$ necessarily has the form $v^{-1}u^{-1}$ with $u \in \mA_\mu(\tau)$, so the opposite containment holds.
\end{proof}

\begin{exa}
    Let $\tau = 651|3|742$, then by Theorem~\ref{thm: mu involution atoms}, the set of inverse atoms is
    \begin{align*}
        \mathcal{A}_{\mu}^{-1}(\tau)= \{6153724, 6153472, 5613724, 5613472\},
    \end{align*}
    since $\mA^{-1}(651) = \{615,561\}$, $\mA^{-1}(3) = \{3\}$, and $\mA^{-1}(742) = \{724,472\}$.
    
    Therefore, the set of atoms is
    \[
    \mA_\mu(\tau) = \{2647315, 2745361, 3647125, 3745126\}.
    \]
\end{exa}

\subsection{From $\mu$ to $\nu$}
\label{ss:mu-to-nu}
In this subsection, we expand a $\mu$--involution Schubert polynomial as a positive and multiplicity-free sum of $\nu$--involution Schubert polynomials when $\nu$ refines $\mu$.
As a consequence of our argument, we also prove a strengthening of Theorem~\ref{t:mu-nu-cover}.

\begin{lem}
    \label{l: split atoms}
    Let $Q$ be a block inverse atom of size $m$, then for any $k \in [m]$, $Q(1) \cdots Q(k)$ and $Q(k+1) \cdots Q(m)$ are both block inverse atoms.
\end{lem}
\begin{proof}
    Let $B$ be the block involution for which $Q$ is a block inverse atom. 
    On the alphabet $L = \{Q(1),\dots ,Q(k)\}$, define $B_L$ where for $(i,j) \in \Cyc(B)$, if $i,j \in L$, then $(i,j)\in \Cyc(B_L)$, and all other elements in $L$ are fixed points in $\Cyc(B_L)$.
    We will show that $Q_L = Q(1) \cdots Q(k)$ is an inverse block atom of $B_L$ using  Theorem~\ref{thm: involution atoms}.
    For any pairs of cycles in $\Cyc(B_L)$ that are also in $\Cyc(B)$, the relative order of their subwords certainly respects the cycle structures since $Q$ is an inverse block atom.
    Therefore we only need to consider pairs of cycles $(i,j),(k,l) \in \Cyc(B)$ where $i \geq j$, $k \geq l$, $i,k \in L$ and at least one of $j,l$ is not in $L$.

\begin{figure}
\begin{center}
        \begin{tabular}{|c|c|c|c|}
            \hline
            $B$ cycles & $Q$ subwords & $Q_L$ subwords & $B_L$ cycles \\ \hline
            \cba{j}{k}{i} & $kij$, \st{$ijk$} & $ki$ & \ab{k}{i} \\ \hline
            \cdab{j}{l}{i}{k} & $ijkl$ & $ijk$ & \cba{j}{k}{l} \\ \hline
            \multirow{7}{*}{\dcba{j}{l}{k}{i}}& $ijkl$ & $ijk$ & \cba{j}{k}{i}\\ \cline{2-4}
            & $klij$ & $kli$ & \bac{l}{k}{i}\\ \cline{2-4}
            & \multirow{3}{*}{$kijl$} & $ki$ & \ab{k}{i}\\ \cline{3-4}
            & & $kij$ & \cba{j}{k}{i}\\ \hline
        \end{tabular}
    \end{center}
    \caption{The cases for two cycles in $B$ with one split when restricting to $B_L$.
    For each possible inverse block atom, the third column depicts the restriction to $Q_L$.
    }
    \label{fig:Q-cycles}
\end{figure}

    We depict these cases in Figure~\ref{fig:Q-cycles},
    omitting the trivial cases where $(i,j),(k,l)$ are completely disjoint ($i \geq j > k \geq l$).
    One can check that the fourth column contains exactly the involutions coming from the first column restricted to the values $\{i,j,k,l\}$ in $B_L$.
    A completely symmetric argument can be made to show that $Q(k+1) \cdots Q(m)$ is an inverse block atom for the block involution obtained from $B$'s restriction to $Q(k{+1})\dots Q(m)$.
    Therefore, the claim follows. 
\end{proof}

\begin{cor}
    \label{c: refine atoms}
    If $\nu$ refines $\mu$, then any inverse atom $w$ in $\mI_\mu$ is also an inverse atom in $\mI_\nu$.
\end{cor}
\begin{proof}
By Theorem~\ref{thm: mu involution atoms}, every inverse atom is a concatenation of its block inverse atoms, so the statement follows from repeated application of Lemma~\ref{l: split atoms}.
\end{proof}

\begin{lem}
    \label{l: mu nu atom containment}
    Let $\pi \in \mI_\mu$ and $\tau \in \mI_\nu$ where $\nu$ refines $\mu$. Then either $ \mA_\nu^{-1}(\tau) \subseteq \mA_\mu^{-1}(\pi)$ or $\mA_\nu^{-1}(\tau) \cap \mA_\mu^{-1}(\pi) = \varnothing$.
\end{lem}
\begin{proof}
    Suppose $w \in \mA_\nu^{-1}(\tau) \cap \mA_\mu^{-1}(\pi)$.
    Then for any $w' \in \mA_\nu^{-1}(\tau)$, Theorem~\ref{thm: mu involution atoms} and Theorem~\ref{t: atom local move} (applied to the standardization of block inverse atoms) imply $w'$ can be obtained from $w$ through a series of transformations from~\eqref{eq:atom-transformation} in the blocks of $\nu$.
    Since $\nu$ refines $\mu$, each such equivalence occurs in a block of $\mu$, so $w' \in \mA_\mu^{-1}(\pi)$, hence $ \mA_\nu^{-1}(\tau) \subseteq \mA_\mu^{-1}(\pi)$.
\end{proof}

A restatement of Lemma~\ref{l: mu nu atom containment} in terms of Bruhat order is:

\begin{cor}
\label{c: intersect = contain}
    For $\pi \in \mI_\mu$ and $\tau \in \mI_\nu$ where $\nu$ refines $\mu$ and $\ell(\nu) = \ell(\mu)+1$, we have $\tau \lessdot^\mu \pi$ if some $w \in \mA_\nu(\tau)$ is also in $\mA_\mu(\pi)$.
\end{cor}

We are now prepared to prove the main result of this subsection.

\begin{thm}
\label{t:mu-nu-expansion}
Let $\pi \in \mI_\mu$, and suppose $\nu$ refines $\mu$, then
    \begin{align*}
        \fS_\pi^\mu  = \sum_{\tau \in \mI_\nu, \, \tau \leq_\mA \pi} \fS_\tau^\nu
    \end{align*}
\end{thm}
\begin{proof}
We show by double containment of inverse atoms.
Let 
\[
S = \bigsqcup_{\tau \in \mI_\nu, \, \tau \leq_\mA \pi}\mA_\nu^{-1}(\tau).
\]
By definition of $\leq_\mA$ we see $S$ is a subset of $\mA_\mu^{-1}(\pi)$.
For any $w \in \mA_\mu^{-1}(\pi)$, $w$ is also an inverse atom of $\tau$ for some $\tau \in \mI_\nu$ by Corollary~\ref{c: refine atoms}.
Then by Corollary~\ref{c: intersect = contain} $\mA_\nu^{-1}(\tau)$ is a subset of $\mA_\mu^{-1}(\pi)$, so $\tau \leq_\mA \pi$, hence $w \in S$.
Therefore $\mA_\mu^{-1}(\pi)$ is a subset of $S$.
\end{proof}

\begin{exa}
    Let $\pi  =  75421|63$, so $\mu = (5,2)$. Then for any composition $\nu$ that refines $(5,2)$, we can expand $\fS_\pi^\mu$ into a positive and multiplicity-free sum of $\nu$-involution Schubert polynomials. For example, suppose that $\nu = (2,3,1,1)$. The inverse atoms of $\pi$
    \begin{align*}
        \mA_\mu^{-1}(\pi) = \{4527163, 5247163, 5271463, 4571263, 5712463, 4715263, 7145263, 7152463\}
    \end{align*}
    are all inverse atoms of some $\nu$-involution from the following set
    \begin{align*}
        B := \{45|721|6|3, 52|741|6|3, 57|124|6|3, 47|152|6|3, 71|542|6|3 \} \subseteq \mI_\nu.
    \end{align*}
    By Lemma~\ref{l: mu nu atom containment}, $B$ is also exactly the set of $\nu$-involutions $\tau$ such that $\tau \leq_{\mA} \pi$. Therefore, by Theorem~\ref{t:mu-nu-expansion},
    \begin{align*}
        \fS_{75421|63}^\mu = \fS_{45|721|6|3}^\nu + \fS_{52|741|6|3}^\nu + \fS_{57|124|6|3}^\nu + \fS_{47|152|6|3}^\nu + \fS_{71|542|6|3}^\nu\, .
    \end{align*}
\end{exa}

\subsection{Transpositions for $\mu$--Bruhat order}
\label{ss:transposition}

We first introduce notation for covers in $\mu$--Bruhat order.
Let $\pi$ and $\tau$ be $\mu$--involutions so that $\pi \lessdot^\mu \tau$.
By Theorem~\ref{t:mu-bruhat-def}, for each $w \in \mA_\mu(\tau)$ there exists $v \in \mA_\mu(\pi)$ so that $v \lessdot w$.
Necessarily, $w = v t_{ij}$ for some $i,j$.
We define transposition like operators on $\mu$--involutions using this relationship.

\begin{defn}
\label{d:mu-bruhat-cover}
    For $\pi$ a $\mu$--involution, if there exists $v \in \mA_\mu(\pi)$ and $\tau$ a $\mu$--involution so that $v \lessdot v t_{ij}$ and $v t_{ij} \in \mA_\mu(\tau)$ we define $t_{ij}^\mu(\pi) = \tau$.
    Otherwise, define $t_{ij}^\mu(\pi) = \pi$. 
\end{defn}

A priori, it is unclear that the operators $t_{ij}^\mu$ are well-defined.
It could be the case that there exist atoms $u,v \in \mA_\mu(\pi)$ so that $u \lessdot ut_{ij}$ and $v \lessdot vt_{ij}$ where $ut_{ij}$ and $vt_{ij}$ are atoms for different $\mu$--involutions.
We show that this cannot occur.

\begin{lem}
    \label{l:mu-bruhat-cover}
    For $\mu \vDash n$ and $i,j \in [n]$, each operator $t_{ij}^\mu: \mI_\mu \to \mI_\mu$ is well-defined.
\end{lem}

\begin{proof}
    Let $\pi \in \mI_\mu$ and $v,v' \in \mA_\mu^{-1}(\pi)$ so that $t_{ij}v$ and $t_{ij}v'$ are also inverse $\mu$--atoms, say for $\sigma,\tau \in \mI_\mu$, respectively.
    The result will follow by showing $\Cyc_\mu(\sigma) = \Cyc_\mu(\tau)$, which we detect by applying the characterization for $\mu$--atoms in Theorem~\ref{thm: mu involution atoms} to $t_{ij}v$ and $t_{ij}v'$.
    Blocks that do not contain $i$ or $j$ are unchanged, so $\sigma,\tau$ agree on these blocks.
    We now consider the blocks containing $i$ and $j$.
    If the values $i,j$ are in the same block $B$, the result follows from Theorem~\ref{t:atom-transposition} since $v|_B$ and $v'|_B$ are inverse block atoms for the same block involution.

    Now let $i$ be in block $B_1$ and $j$ be in block $B_2$ of $\pi$.
    Consider $c_1, c_2 \in \Cyc_{B_1}(\pi)$.
    In $v$ and $v'$, the possible subwords of the values in these cycles is determined by Theorem~\ref{thm: involution atoms}.
    Since $t_{ij}v$ and $t_{ij}v'$ are inverse $\mu$--atoms, the subwords for these same values correspond to pairwise cycle structures in their associated $\mu$--involutions.
    We must show the resulting pairwise cycle structures are the same.

    When $c_1,c_2$ do not contain $i$, their corresponding subwords are unchanged.
    Therefore, we can reduce to the case where $c_1 = (i,h)/(h,i)$ and $c_2 = (k,l)$.
    Further, when $c_1$ and $c_2$ do not nest the values $h,i,k,l$ form a unique subword in $v$ and $v'$, so $t_{ij}v$ and $t_{ij}v'$ have the same subword for values $h,j,k,l$.
    Therefore, the cycles containing these values in $\sigma$ and $\tau$ must be the same.
    The possible cases where $c_1$ and $c_2$ nest are treated in Figure~\ref{f:nested-mu-covers}.
    Since $v \lessdot t_{ij}v$, only certain subwords can appear in the third column.
    Note that all the subwords in the fourth column correspond to the unique cycle structure in the fifth column.
    By taking the reverse complement, the same argument holds for $B_2$ so result follows.
\end{proof}

\begin{figure}
\small
\begin{tabular}{ |c||c|c|c|c| } 
\hline
Relative order & $\pi$--cycles & $\mA_\mu^{-1}(\pi)$--subwords & $\mA_\mu^{-1}(\tau)$--subwords & $\tau$--cycles\\
\hline
$l < h = i < k < j$ & $\cba{l}{i}{k}$ & $kli$, \st{$ikl$} & $klj$ & $\bac{l}{k}{j}$\\ 
\hline
$l < h < i < k < j$ &  $\dcba{l}{h}{i}{k}$ & $klih$,\st{$iklh$},\st{$ihkl$} & $kljh$ & $\cdab{l}{h}{k}{j}$\\
\hline
$i<l = k< j <h$ & $\cba{i}{k}{h}$ & $khi$, \st{$hik$} & $khj$ & $\acb{k}{j}{h}$\\
\hline
$i<l = k< h<j $ & $\cba{i}{k}{h}$  & $khi$, \st{$hik$} & $khj$ & \abc{k}{h}{j}\\
\hline
$i<l<j < k <h$ &  $\dcba{i}{l}{k}{h}$ & $klhi$,\st{$khil$},\st{$hikl$} & $klhj$ & $\cdab{l}{j}{k}{h}$\\
\hline
$i<l < k< j <h$ &  $\dcba{i}{l}{k}{h}$ & $klhi$,\st{$khil$},\st{$hikl$} & $klhj$ & $\badc{l}{k}{j}{h}$\\
\hline
$i<l < k <h<j$ &  $\dcba{i}{l}{k}{h}$ & $klhi$,\st{$khil$},\st{$hikl$} & $klhj$ & $\bacd{l}{k}{h}{j}$\\
\hline
$l<i<h<j<k$ &  $\dcba{l}{i}{h}{k}$ & $hkli,klhi,hikl$ & $hjkl,hklj,klhj$ & $\dbca{l}{h}{j}{k}$\\
\hline
$l<i<h<k<j$ &  $\dcba{l}{i}{h}{k}$ & $hkli,klhi$, \st{$hikl$} & $hklj,klhj$ & $\cbad{l}{h}{k}{j}$\\
\hline
\end{tabular}
\caption{The action when $t^\mu_{ij}(\pi) = \tau \neq \pi$ for nested cycles containing $i$ when $i,j$ are in different blocks in $\pi$.
The third column depicts the possible subwords of the values in these cycles in $v\in \mA_\mu^{-1}(\pi)$ where $t_{ij}v \in \mA_\mu^{-1}(\tau)$, with crossed out subwords excluded by the fact that $v \lessdot t_{ij}v$.
The fourth column depicts the corresponding subwords in $t_{ij}v$.
We omit cases where there are no intermediate values between $i$ and $j$.
\label{f:nested-mu-covers}}
\end{figure}

Note Figure~\ref{f:nested-mu-covers} gives a partial combinatorial description of the  operators $t_{ij}^\mu$ acting on the cycles of a $\mu$--involution when $i,j$ are in different blocks.
In Figure~\ref{f:other-mu-covers}, we complete this characterization.
For $i,j$ in the same block, the operators acts as for involutions so the combinatorial description appears in Figure~\ref{fig:I-bruhat-cover}.

Using the $t^\mu_{ij}$ operators, we obtain a direct analogue of Theorem~\ref{t:atom-transposition} for $\mu$--involutions.
\begin{theorem}
        \label{t:mu-atom-transposition}
        Let $\pi, \tau \in \mI_\mu$ such that $\pi \lessdot^\mu \tau$, then
        \[
        \mA_\mu(\tau) = \{w t_{ij}: w \in \mA_\mu(\pi),\, t_{ij}^\mu(\pi) = \tau, w \lessdot wt_{ij}\}.
        \]
    \end{theorem}

    \begin{proof}
        By Lemma~\ref{l:mu-bruhat-cover}, the right-hand side is contained in $\mA_\mu(\tau)$.
        The opposite containment follows by the subword characterization of $\leq^\mu$ from Theorem~\ref{t:mu-bruhat-def}.
    \end{proof}

\begin{rem}
    One can extend the definition of $t^\mu_{ij}$ to transposition-like operators $v_{ij}$ on $\mu$--involutions as follows.
    For $\tau \in \mI_\mu$, let $v_{ij}$ act on $\tau$ by exchanging the values $i,j$ in $\Cyc_\mu(\tau)$, then breaking any resulting pairs that are increasing into two fixed points, reordering the cycles in each block based on the $\mu$-cycle order.
    To be consistent with the operators $t^\mu_{ij}$ when they act non-trivially, we additionally require for $a < b < c <d$ in the same block that $v_{ad}$ sends the cycle pair $(b,a),(d,c)$ to $(d,a),(b,b),(c,c)$ and the cycle pair $(c,a),(d,b)$ to $(d,a),(c,b)$.
    Similarly, for $a<b<c$ in the same block, $v_{ac}$ sends the cycle pairs $(b,a),(c,c)$ and $(a,a),(c,b)$ to $(c,a)(b,b)$.
\end{rem}
    
\begin{exa}
     Let $\tau = 651|3|742$. It's cycle notation is $(6,1)(5,5)|(3,3)|(7,2)(4,4)$.
     \begin{align*}
         t^\mu_{15}(\tau) =v_{15}(\tau) &= (6,5)(1,1)|(4,4)|(7,2)(3,3) = 165|4|732\\
         t^\mu_{34}(\tau) =v_{34}(\tau) &= (6,1)(5,5)|(4,4)|(7,2)(3,3) = 651|4|732
     \end{align*}
     Additionally, we have $t^\mu_{25}(\tau) = \tau$ but
     \begin{align*}
        t^\mu_{25} (\tau) \neq  v_{25}(\tau) &= (6,1)(2,2)|(3,3)|(4,4)(7,5) = 621|3|475.\\
     t^\mu_{17}(\tau) \neq v_{17}(\tau) &= (5,5)(6,6)(7,7)|(3,3)|(1,1)(2,2)(4,4) = 567|3|124.
     \end{align*}
Similarly, $\pi = 7|318945|62$ has cycle notation is $(7,7)|(3,1)(8,4)(9,5)|(6,2)$.
We compute
     \begin{align*}
         t^\mu_{18}(\pi) =v_{18}(\pi) &= (7,7)|(8,1)(3,3)(4,4)(9,5)|(6,2) = 7|834915|62,\\
         t^\mu_{49}(\pi) = v_{49}(\pi) &= (7,7)|(3,1)(9,4)(8,5)|(6,2) = 7|319854|62.
     \end{align*}
\end{exa}
\begin{exa}
    The $\mu$-Bruhat covers allow us to compute (inverse) atoms recursively using Theorem~\ref{t:atom-transposition}.
    Let $\pi=2|176|543$ and $\tau=5|176|234$.
    Then $\pi \lessdot^\mu \tau$ with $t_{25}^\mu(\pi) = \tau$ being the only transposition that connects these two.
    Therefore, since 
    \[
    \mA_\mu^{-1}(\pi) = \{2176453,2176534\},
    \]
    we see $\mA_\mu^{-1}(\tau) = \{5176234\}$ by applying $t_{25}$ on the left to $2176453,2176534$ and keeping only the ones that give Bruhat covers.
\end{exa}

\begin{figure}
    \small
    \begin{tabular}{ |c||c|c|c|c| } 
    \hline
    Relative order & $\pi$--cycles & $\mA_\mu^{-1}(\pi)$--subwords & $\mA_\mu^{-1}(\tau)$--subwords & $\tau$--cycles\\
    \hline
    $l=k<i<h<j$ & $\acb{k}{i}{h}$ & $khi$ & $khj$ & $\abc{k}{h}{j}$\\ 
    \hline
    $l<k<i<h<j$ & $\badc{l}{k}{i}{h}$ & $klhi$ & $klhj$ & $\bacd{l}{k}{h}{j}$\\ 
    \hline
    $l<i<k<j<h$ & $\cdab{l}{i}{k}{h}$ & $klhi$ & $klhj$ & $\badc{l}{k}{j}{h}$\\ 
    \hline
    $l<i<k<h<j$ & $\cdab{l}{i}{k}{h}$ & $klhi$ & $klhj$ & $\bacd{l}{k}{h}{j}$\\ 
    \hline
    \end{tabular}
    \caption{Non-trivial actions of $t_{ij}^\mu$ for $i,j$ in different blocks for pairs of cycles not nested.
    In this chart, we only depict cases where the pairwise structure of the cycle containing $i$ or its counterpart changes.}
    \label{f:other-mu-covers}
\end{figure}

Using Theorem~\ref{thm: mu involution atoms}, we prove an analogue of the Exchange Property for $\mu$--involutions.
For $\bfa = (a_1,\dots, a_p)$ a word, say $(\bfa,i)$ is \newword{nearly $\mu$--reduced} for $\tau \in \mI_\mu$ if $\bfa$ is not $\mu$--reduced and $(a_1,\dots,\widehat{a_i}, \dots, a_p) \in \mR_\mu(\tau)$. 
We first show nearly $\mu$--reduced words that are reduced fail to be $\mu$--reduced within a single block.

\begin{lem}
\label{l:nearly-reduced-same-block}
    Let $(\bfa,j)$ be nearly $\mu$--reduced for $\tau \in \mI_\mu$ so that $\bfa \in \mR(w)$ for some $w \in S_n$.
    Then there exists $v \in \mA_\mu(\tau)$ with $v = w t_{ik} \lessdot w$ so that $i,k$ are in the same block of $\tau$.
    
\end{lem}

\begin{proof}
    Since $(\bfa,j)$ is nearly $\mu$--reduced for $\tau$, $v = s_{a_1} \dots \widehat{s_{a_j}} \dots s_{a_p} \in \mA_\mu(\tau)$.
    Then for $t_{ik} = s_{a_p} \dots s_{a_{j+1}} s_{a_j} s_{a_{j+1}} \dots s_{a_p}$ we see $w = vt_{ik}$.
    If positions $i$, $k$ were in different blocks of $\mu$, then by Theorem~\ref{thm: mu involution atoms} we see $v(i)$ and $v(k)$ are in different blocks of $\mu$.
    Since $\bfa$ is not $\mu$--reduced, this implies $v(i) > v(k)$ so $w < v$, hence $\bfa$ is not reduced.
    However, by assumption $\bfa$ is reduced so $i,k$ must be in the same block.
\end{proof}

We now prove our exchange lemma for $\mu$--involutions.
\begin{thm}
    \label{t:mu-exchange}
    Let $(\bfa,m)$ be nearly $\mu$--reduced for $\tau \in \mI_\mu$.
    Then there exists unique $k \neq m$ so that $(\bfa,k)$ is nearly $\mu$--reduced for $\tau$.
\end{thm}

\begin{proof}
    Let $v = s_{a_1} \dots \widehat{s_{a_m}} \dots s_{a_p} \in \mA^{-1}_\mu(\tau)$.
    If $\bfa$ is not reduced, by Lemma~\ref{l:perm-exchange} (Deletion Property), there exists $k \neq m$ so that $(\bfa,k)$ is nearly reduced for $v$.
    In fact, $k$ is unique so the result follows.

    If $\bfa \in \mR(w)$ for $w \in S_n$, Lemma~\ref{l:nearly-reduced-same-block} shows $w = t_{ij} v $ with $i,j$ in the same block $B$ of $\mu$.
    By Theorem~\ref{t:atoms-exchange} and the definition of Bruhat covers, there exist unique $i'<j'$ so that 
    \begin{equation}
        \label{eq:transposition-pair}
        v \mid_B \; \neq t_{i,j}t_{i',j'} (v\mid_B) = (t_{i,j} t_{i',j'
    }v) \mid_B \in \mA^{-1}(\tau \mid_B).
    \end{equation}
    Then by Theorem~\ref{thm: mu involution atoms} we see $v' = t_{i,j} t_{i',j'
    }v \in \mA^{-1}_\mu(\tau)$. Furthermore, $v' \lessdot t_{ij}v$, so there exists $k$ such that 
    $v' = s_{a_1} \dots \widehat{s_{a_k}} \dots s_{a_p}$.
    Therefore, $(\bfa,k)$ is nearly $\mu$--reduced for $\tau$.
    Since $i',j'$ can be identified from $k$, their uniqueness implies the uniqueness of $k$.
\end{proof}

The following is a strengthening of Theorem~\ref{t:mu-exchange} mirroring~\cite{hamaker2018transition}*{Thm.~3.23}.

\begin{cor}
\label{c:strong-mu-exchange}
    Fix $\tau \in \mI_\mu$ and recall $y_\tau$ is the involution so that $\Cyc_\mu(\tau) = \Cyc(y_\mu)$.
    Let $v \in \mA_\mu(\tau)$ and $w$ \textbf{not} a $\mu$--atom so that $v \lessdot w = v t_{ij}$.
    Then $i < j$ are in the same block $B$ of $\tau$.
    Furthermore, there exist unique numbers $i' < j'$ so that $v' = w t_{i'j'} \in \mA_\mu(\tau)$ satisfying $i' \in \{j,y_\tau(j)\}$ and $j' \in \{i,y_\tau(i)\}$.
\end{cor}

\begin{proof}
    The existence and uniqueness of $i',j'$ is equivalent to Theorem~\ref{t:mu-exchange}.
    From its proof, we see $i,j$ are in the same block.
    The final part is~\cite{hamaker2018transition}*{Thm.~3.23~(b)} applied to $B$.
\end{proof}

The precise values of $i' <j'$ in Corollary~\ref{c:strong-mu-exchange} are specified in Table 2 of~\cite{hamaker2018transition}.
 
\subsection{Transitions for $\mu$--involutions}

We conclude this section by proving transition equations for $\mu$--involution Schubert polynomials.
To do so, we require several definitions.
Let $\mu' = (\mu_1, \dots, \mu_k, 1)$ be the composition obtained from appending $1$ at the end of $\mu$.
Define $\tau'$ by appending $n+1$ as a fixed point to $\tau$, so $\tau' \in \mI_{\mu'}$.
In the remainder of this section, we identify $\tau$ with $\tau'$ as they have the same set of atoms by Theorem~\ref{thm: mu involution atoms}.
This identification is required to obtain all the necessary Bruhat covers in the subsequent definitions, as demonstrated in Example~\ref{e: transition}.

For $\mu \vDash n$, $\tau \in \mI_\mu$, and $i,j$ not necessarily distinct, let
    \begin{align*}
        \Phi(\tau;i,j) = \{t^\mu_{a,b}(\tau') \neq \tau\, : \,  a \in \{i,j\}\}, \\
        \Psi(\tau;i,j) = \{t^\mu_{a,b}(\tau') \neq \tau \, : \,  b \in \{i,j\}\}.
    \end{align*}
We now state the transition equations for $\mu$--involution Schubert polynomials.
    \begin{thm}
    \label{t:mu-transition}
        Let $\tau\in \mI_\mu$.
        For $(i,j) \in \Cyc_\mu(\tau)$,
        \begin{align*}
            2^{- \delta(i,j)} (x_i+x_j)\fS_\tau^\mu =   \sum_{\sigma \in \Phi(\tau;i,j)}\fS_\sigma^\mu - \sum_{\sigma \in \Psi(\tau;i,j)}\fS_\sigma^\mu.
        \end{align*}
        
    \end{thm}
    
    \begin{proof}
        We first define the following two sets of permutations:
        \begin{align*}
            A := \bigcup_{v \in \mA_\mu(\tau)} \{u \in S_n \, : \, v \lessdot u, u = vt_{ir} \text{ or }u = vt_{jr}\},\\
            C := \bigcup_{v \in \mA_\mu(\tau)} \{u \in S_n \, : \, v \lessdot u, u = vt_{ri} \text{ or }u = vt_{rj}\}.
        \end{align*}
        Viewed as multisets, there is no multiplicity in $A$ or $C$.
        To see this for $A$, given $v_1,v_2 \in \mA_\mu^{-1}(\tau)$ and any $r$, $t_{ir}v_1 = t_{jr}v_2$ cannot occur as $i$ must appear before $j$ in any inverse atom of $\tau$.
        Therefore any multiplicity is of the form $t_{ir}v_1 = t_{js}v_2$ where $r \neq s$.
        In this case, the subword formed by $\{i,j,r,s\}$ must be $irsj$ in $v_1$ and $rijs$ in $v_2$. However, since $r>i>j$, $r$ cannot appear to the left of $i,j$, so $rijs$ cannot be a subword in any inverse atom, a contradiction.
        A similar argument holds for $C$.
        
        By Theorem~\ref{t:monk}, the left hand side expands as
        \begin{align*}
            2^{-\delta(i,j)} (x_i+x_j)\fS_\tau^\mu &= 2^{-\delta(i,j)}\sum_{v \in \mA_\mu(\tau)}(x_i+x_j) \fS_v =  \sum_{u \in A} \fS_u - \sum_{u \in C} \fS_u.
        \end{align*}
        Now define $A' := \{u \in \mA_\mu(\sigma) : \sigma \in \Phi(\tau;i,j)\}$ and $C' := \{w \in \mA_\mu(\sigma) : \sigma \in \Psi(\tau;i,j)\}$.
        Note $A' \subseteq  A$ and $C' \subseteq C$ by Theorem~\ref{t:mu-atom-transposition}.
        Then the right hand side is
        \[
        \sum_{u \in A'} \fS_u - \sum_{w \in C'} \fS_w.
        \]
        Then the result will follow by showing $A \setminus A'  = C \setminus C'$.
        
        Let $w \in A \setminus A'$.
        By the definition of $A$, there exists $v \in \mA_\mu(\tau)$ so that $w = vt_{ab}$ with $a < b$ and $a \in \{i,j\}$.
        By Corollary~\ref{c:strong-mu-exchange}, we see $a,b$ are in the same block of $\tau$, hence $y_\tau(a)$ and $y_\tau(b)$ as well.
        Furthermore, the result guarantees there exists $a'<b'$ with $u = wt_{a'b'}$ so that $b' \in \{i,j\}$.
        Therefore, $w \in C \setminus C'$.
        By a symmetric argument, every $w \in C \setminus C'$ is also in $A \setminus A'$, so the result follows.
    \end{proof}

\begin{cor}
    \label{c:two-block-transition}
    The conclusion of Theorem~\ref{t:mu-transition} also holds for $(i,i)$ when $i$ is in a block of size 2.
\end{cor}

\begin{proof}
    We can split the block containing $i$ into two blocks of size one without modifying the sets $A,A',C,C'$ from the proof of Theorem~\ref{t:mu-transition}.
\end{proof}

\begin{exa}
    \label{e: transition}
    Let $\tau = 5|4613|72|8$, then $\Cyc_\mu(\tau) = \{(5,5),(4,1),(6,3),(7,2),(8,8)\}$. It's $\mu$-Bruhat covers are given by the following transpositions:
    \begin{align*}
        t_{12}^\mu(\tau) &= 5|4623|71|8, \quad t_{28}^\mu(\tau) = 5|4613|78|2, \quad t_{37}^\mu(\tau) = 5|4167|32|8, \quad t_{56}^\mu(\tau) = 6|4513|72|8,\\ 
        \quad t_{67}^\mu(\tau) &= 5|4713|62|8, \quad t_{78}^\mu(\tau) = 5|4613|82|7, \quad t_{13}^\mu(\tau) = t_{16}^\mu(\tau) = t_{46}^\mu(\tau) = 5|6431|72|8.
    \end{align*}
    Then by Theorem~\ref{t:mu-transition},
    \begin{align*}
        x_5 \cdot \fS_\tau^\mu &= \fS_{t_{56}^\mu(\tau)}^\mu\\
        (x_4 + x_1)\fS_\tau^\mu &= \fS_{t_{12}^\mu(\tau)}^\mu + \fS_{t_{13}^\mu(\tau)}^\mu\\
        (x_7 + x_2)\fS_\tau^\mu &= \fS_{t_{28}^\mu(\tau)}^\mu + \fS_{t_{78}^\mu(\tau)}^\mu - \fS_{t_{12}^\mu(\tau)}^\mu - \fS_{t_{37}^\mu(\tau)}^\mu - \fS_{t_{67}^\mu(\tau)}^\mu\\
        (x_6 + x_3)\fS_\tau^\mu &= \fS_{t_{37}^\mu(\tau)}^\mu + \fS_{t_{67}^\mu(\tau)}^\mu - \fS_{t_{13}^\mu(\tau)}^\mu - \fS_{t_{56}^\mu(\tau)}^\mu
    \end{align*}
    Since $\{7,2\}$ are in a block of size $2$, Corollary~\ref{c:two-block-transition} also gives
    \begin{align*}
        x_2 \cdot \fS_\tau^\mu &= \fS_{t_{28}^\mu(\tau)}^\mu - \fS_{t_{12}^\mu(\tau)}^\mu\\
        x_7 \cdot \fS_\tau^\mu &= \fS_{t_{78}^\mu(\tau)}^\mu - \fS_{t_{37}^\mu(\tau)}^\mu - \fS_{t_{67}^\mu(\tau)}^\mu
    \end{align*}
    However, since $\{4,6,1,3\}$ are not in a block of size $2$ and are not fixed points in their block, Theorem~\ref{t:mu-transition} does not apply to the polynomials $x_i \cdot \fS_\tau^\mu$ for $i \in \{4,6,1,3\}$. In fact, one can check that these polynomials are not even in the $\ZZ$ span of $\mu$-involution Schubert polynomials for $\mu = (1,4,2,1)$, so it is impossible to expand them in this case.
\end{exa}

\section{Future directions}

The results of this paper suggest several natural directions of further study, both for combinatorial and geometric aspects of the theory of $\mu$–involutions and their associated Schubert polynomials.

\subsection{Shellability of $\mu$–Bruhat order}

A natural question to ask is whether the $\mu$–Bruhat order is thin and/or shellable.
Thinness is a key structural property often appearing in posets arising from Coxeter theory and geometry, while shellability shows that the order complex is homotopic to a wedge of spheres.
The two together would provide a combinatorial description of the M\"obius function of the poset.
Such results would parallel the well-known properties of Bruhat order on permutations and on involutions~\cites{edelman1981bruhat,incitti2004bruhat}.

\subsection{$\mu$–involution Stanley symmetric functions}

Another natural direction is to define and study Stanley symmetric functions associated to $\mu$–involutions.
Given $\pi \in I_\mu$, one may define
\[
F^\mu_\pi = \sum_{w \in A_\mu(\pi)} F_w
\]
where $F_w$ is the classical Stanley symmetric function from~\cite{stanley1984number} associated to a permutation $w$.
By construction, the coefficient in $F^\mu_\pi$ of $x_1 \dots x_{\ell_\mu(\pi)}$ is $|\mR_\mu(\pi)|$, the number of reduced $\mu$--involution words for $\pi$.

This construction interpolates between two well-studied cases.
When $\mu=(1^n)$, we recover the usual Stanley symmetric functions, which are known to be Schur positive~\cite{edelman1987balanced}.
When $\mu=(n)$, we obtain the involution Stanley symmetric functions studied in~\cite{Hamaker2017SchurPA}, which are Schur-$P$ positive.

\begin{question}
    Can one define a notion of positivity interpolating between Schur and Schur-P positivity that the $F^\mu_\pi$ satisfy?
\end{question}

\subsection{Combinatorial formulas for $\mu$--involution Schubert polynomials}

Schubert polynomials exhibit rich combinatorial structure through their many combinatorial formulas, including the famous pipedream~\cite{bb93,bjs93,fk96} and bumpless pipedream~\cite{lls21,weigandt21} formulas.
An analogous pipedream formula for involution Schubert polynomials was also introduced in ~\cite{hamaker2022involution}.

\begin{question}
    Is there a pipedream like formula for $\mu$--involution Schubert polynomials generalizing these two formulas?
    Or a bumpless pipedream formula?
\end{question}

One possible approach is through the framework of co-transition or transition formulas, which have been successful in other settings.
However, preliminary computations indicate that the co-transition approach does not behave as cleanly for $\mu$–involutions as it does for permutations or involutions.
The crux of the strategy is to construct a sequence of transition equations where each $\Psi$ set is empty, eventually terminating with the base case $w_0$.
In fact, for the permutation and involution cases any such sequence with $n$ fixed will arrive at $w_0$.
The following example shows that such an approach requires some modification if it is to be applied for $\mu$--involutions.
\begin{exa}
    Let $\tau = 653|421|7$ with $\Cyc(\tau) = \{(6,3),(5,5),(4,2),(1,1),(7,7)\}$.
    Using Theorem~\ref{t:mu-transition}, we have the following sequence of products with $\Psi(\tau;i,j)$ empty:
    \begin{align*}
        x_2 \cdot \fS_\tau^\mu &= \fS_{t_{27}^\mu(\tau)}^\mu = \fS_{653|417|2}^\mu,\\
        (x_3 + x_6)\fS_{653|417|2}^\mu &= \fS_{763|421|5}^\mu + \fS_{647|321|5}^\mu,\\
        x_2 \cdot \fS_{763|421|5}^\mu &= \fS_{763|415|2}^\mu.
    \end{align*}
    For $\pi = 763|415|2$ and every $(i,j) \in \Cyc_\mu(\pi)$, the resulting transition equation either requires the auxiliary value $8$ (so $t_{i8}^\mu$ is applied) or has nonempty $\Psi(\pi; i ,j)$. 
\end{exa}

In particular, it is unclear whether the cotransition outline will terminate with $w_0$ starting with an arbitrary $\mu$--involution.
Therefore, the upward induction on $\ell^\mu$ does not have a clear base case.

\subsection{K-theoretic generalizations}

Finally, it is natural to ask whether the theory developed here extends to $K$–theory of Borel orbits in $Q_n$.
Grothendieck polynomials $\fG_w$ represent $K$--theory classes of Schubert varieties, so are the $K$-theoretic analogues of Schubert polynomials.
We define $K$–theoretic analogues of $\mu$–involution Schubert polynomials as a sum of certain Grothendieck polynomials.
For $\pi \in \mI_\mu$, the \newword{Hecke atoms} of $\pi$ are $\mathcal{B}_\mu(\pi) = \{w: w^{-1} \muprod w = \pi\}$.
The \newword{$\mu$--involution Grothendieck polynomial} indexed by $\pi$ is
\begin{align*}
    \fG_\pi^\mu = \sum_{w \in \mathcal{B}_\mu(\pi)} \fG_w.
\end{align*}
It is not clear whether these polynomials have any geometric significance.
Even in the involution case, this is unknown.
The powers of 2 required to obtain Wyser and Yong's representatives from involution Schubert polynomials do not have a neat translation to the $K$--theory setting~\cite{marberg2024some}.
We hope $\mu$--involution Grothendieck polynomials can shed some light on the challenge of describing $K$--theory classes for $O_n$--orbit closures combinatorially.

\bibliographystyle{alpha}
\bibliography{citation}{}
\end{document}